\documentclass[10pt]{elsarticle}
\usepackage{geometry, enumerate, amsthm, amsmath, amssymb, amscd, verbatim, color, graphicx}
\usepackage{hyperref}
\usepackage{courier}
\usepackage[all]{xy}

\numberwithin{equation}{section}

\newtheorem{Prop}[equation]{Proposition}
\newtheorem{Thm}[equation]{Theorem}
\newtheorem{Lem}[equation]{Lemma}
\newtheorem{Cor}[equation]{Corollary}
\theoremstyle{definition}\newtheorem{Def}[equation]{Definition}
\newtheorem{Defs}[equation]{Definitions}
\newtheorem{Ex}[equation]{Example}
\newtheorem{Rem}[equation]{Remark}
\theoremstyle{definition}

\newcommand{\N}{\mathbb{N}}
\newcommand{\Q}{\mathbb{Q}}
\newcommand{\Z}{\mathbb{Z}}
\newcommand{\F}{\mathbb{F}}
\newcommand{\Int}{\textnormal{Int}}

\newcommand{\mfm}{\mathfrak{m}}

\newcommand{\whA}{\widehat{A}}
\newcommand{\whB}{\widehat{B}}
\newcommand{\whD}{\widehat{D}}
\newcommand{\whF}{\widehat{F}}
\newcommand{\whK}{\widehat{K}}
\newcommand{\whP}{\widehat{P}}

\newcommand{\whT}{\widehat{T}}

\newcommand{\whm}{\widehat{\mathfrak{m}}}

\newcommand{\mcD}{\mathcal{D}}

\journal{Journal of Pure and Applied Algebra}

\begin{document}

\begin{frontmatter}

\title{Decomposition of Integer-valued Polynomial Algebras}

\author{Giulio Peruginelli}
\ead{gperugin@math.unipd.it}
\address{Department of Mathematics, University of Padova, Via Trieste, 63
35121 Padova, Italy}

\author{Nicholas J. Werner}
\ead{wernern@oldwestbury.edu}
\address{Department of Mathematics, Computer and Information Science, SUNY College at Old Westbury, P.O.\ Box 210, Old Westbury, NY 11568}

\begin{abstract}
\noindent Let $D$ be a commutative domain with field of fractions $K$, let $A$ be a torsion-free $D$-algebra, and let $B$ be the extension of $A$ to a $K$-algebra. The set of integer-valued polynomials on $A$ is $\Int(A) = \{f \in B[X] \mid f(A) \subseteq A\}$, and the intersection of $\Int(A)$ with $K[X]$ is $\Int_K(A)$, which is a commutative subring of $K[X]$. The set $\Int(A)$ may or may not be a ring, but it always has the structure of a left $\Int_K(A)$-module. 

A $D$-algebra $A$ which is free as a $D$-module and of finite rank is called $\Int_K$-decomposable if a $D$-module basis for $A$ is also an $\Int_K(A)$-module basis for $\Int(A)$; in other words, if $\Int(A)$ can be generated by $\Int_K(A)$ and $A$. A classification of such algebras has been given when $D$ is a Dedekind domain with finite residue rings. In the present article, we modify the definition of $\Int_K$-decomposable so that it can be applied to $D$-algebras that are not necessarily free by defining $A$ to be $\Int_K$-decomposable when $\Int(A)$ is isomorphic to $\Int_K(A) \otimes_D A$. We then provide multiple characterizations of such algebras in the case where $D$ is a discrete valuation ring or a Dedekind domain with finite residue rings. In particular, if $D$ is the ring of integers of a number field $K$, we show that an $\Int_K$-decomposable algebra $A$ must be a maximal $D$-order in a separable $K$-algebra $B$, whose simple components have as center the same finite unramified Galois extension $F$ of $K$ and are unramified at each finite place of $F$. Finally, when both $D$ and $A$ are rings of integers in number fields, we prove that $\Int_K$-decomposable algebras correspond to unramified Galois extensions of $K$.\\


\noindent MSC Primary 13F20 Secondary 16H10, 11C99
\end{abstract}
\end{frontmatter}

\section{Introduction}
Let $D$ be a commutative integral domain with field of fractions $K$. The ring of integer-valued polynomials over $D$ is defined to be $\Int(D) := \{f \in K[X] \mid f(D) \subseteq D\}$. The ring $\Int(D)$, its elements, and its properties have been popular objects of study over the past several decades and continue to be so today. The book \cite{CaCh} is the standard reference on the topic.

Beginning around 2010, attention turned to polynomials that are evaluated on $D$-algebras rather than on $D$ itself. This can be seen in the work of Evrard, Fares and Johnson \cite{EFJ, EvrJoh}, Frisch \cite{Fri1, Fri1Corr, Fri2, FriTriangular}, Loper \cite{LopWer}, Peruginelli \cite{ChabPer, HeidaryanLongoPer, Per1, PerDivDiff, PerFinite, PerWer}, Werner \cite{Wer, Wer2, Wer1}, and Naghipour, Rismanchian, and  Sedighi Hafshejani \cite{NagRisSed}. A good example of these new rings of integer-valued polynomials comes from considering the polynomials in $K[X]$ that map each element of the matrix algebra $M_n(D)$ back to $M_n(D)$. 

\begin{Ex}\label{M_n(K) example}
Associate $K$ with the scalar matrices in $M_n(K)$. Then, for any polynomial $f(X) = \sum_{i=0}^t q_i X^i \in K[X]$ and any matrix $a \in M_n(D)$, we can evaluate $f$ at $a$ to produce the matrix $f(a) = \sum_{i=0}^t q_i a^i$. If $f(a) \in M_n(D)$ for each $a \in M_n(D)$, then $f$ is said to be integer-valued on $M_n(D)$. The set of all such polynomials is denoted by 
$$\Int_K(M_n(D)) := \{f \in K[X] \mid f(M_n(D)) \subseteq M_n(D)\},$$
and it is easy to verify that $\Int_K(M_n(D))$ is a subring of $K[X]$. 

We can form a larger collection of polynomials that are integer-valued on $M_n(D)$ by considering polynomials whose coefficients come from $M_n(K)$ rather than from $K$. That is, we form the set 
$$\Int(M_n(D)) := \{f \in M_n(K)[X] \mid f(M_n(D)) \subseteq M_n(D)\}.$$
Since $M_n(K)$ is noncommutative, we follow standard conventions regarding polynomials with non-commuting coefficients, as in \cite[\S 16]{Lam1}. In $M_n(K)[X]$, 
we assume that the indeterminate $X$ commutes with each element of $M_n(K)$, and we define evaluation to occur when the indeterminate is to the right of any coefficients. So, given $f(X) = \sum_{i=0}^t q_i X^i \in M_n(K)[X]$, we consider $f(X)$ to be equal to $\sum_{i=0}^t X^i q_i$ as an element of $M_n(K)[X]$, but to evaluate $f(X)$ at a matrix $a \in M_n(D)$, we must first write $f(X)$ in the form $f(X) = \sum_{i=0}^t q_i X^i$, and then $f(a) = \sum_{i=0}^t q_i a^i$. A consequence of this is that evaluation is no longer a multiplicative homomorphism; that is, if $f(X) = g(X) h(X)$ in $M_n(K)[X]$, then it may not be true that $f(a)$ equals $g(a) h(a)$. Because of this difficulty, it is not clear whether $\Int(M_n(D))$ is closed under multiplication. Despite the complications associated with evaluation of polynomials in this setting, one may prove that $\Int(M_n(D))$ is a (noncommutative) subring of $M_n(K)[X]$ \cite[Thm. 1.2]{WerMat}. Thus, we are able to construct a noncommutative ring of integer-valued polynomials. 

We can actually say more. In \cite[Thm. 7.2]{Fri1}, Sophie Frisch proved that $\Int(M_n(D))$ is itself a matrix ring. Specifically, $\Int(M_n(D)) \cong M_n(\Int_K(M_n(D)))$, where the isomorphism is given by associating a polynomial with matrix coefficients to a matrix with polynomial entries. (This isomorphism is the restriction of the classical isomorphism between the polynomial ring $M_n(K)[X]$ and the matrix ring $M_n(K[X])$). Because of Frisch's theorem, many questions about $\Int(M_n(D))$ can be reduced to questions about $\Int_K(M_n(D))$, and the latter ring---being commutative---is usually easier to work with.
\end{Ex}

Broadly speaking, the point of this paper is to study the relationship between a commutative ring of integer-valued polynomials such as $\Int_K(M_n(D))$ and its extension $\Int(M_n(D))$. In particular, we wish to determine when and how Frisch's theorem \cite[Thm. 7.2]{Fri1} can be generalized to algebras other than matrix rings. While matrix rings will be prominent in our work, the majority of our theorems deal with  general algebras. However, our basic definitions are inspired by the situation described in Example \ref{M_n(K) example}. 

We begin by giving notation and conventions for working with polynomials over algebras. As before, let $D$ be a commutative integral domain with field of fractions $K$. Let $A$ be a torsion-free $D$-algebra and take $B = K \otimes_D A$ to be the extension of $A$ to a $K$-algebra. We associate $K$ and $A$ with their canonical images in $B$ via the maps $k \mapsto k \otimes 1$ and $a \mapsto 1 \otimes a$. Much of our work will involve polynomials in $B[X]$. The algebra $B$ may be noncommutative, but we will assume that $X$ commutes with all elements of $B$. Moreover, we define evaluation of polynomials in $B[X]$ at elements of $A$ just as we did in Example \ref{M_n(K) example} where $A = M_n(D)$ and $B = M_n(K)$. Given  $f(X)=\sum_{i=0}^t c_i X^i \in B[X]$ and $b \in B$, we define 
\begin{equation*}
f(b): = \sum_{i=0}^t c_i b^i.
\end{equation*}
Note that the map $B[X] \to B$ given by evaluation at $b$ is not a multiplicative homomorphism unless $b$ lies in the center of $B$.

Finally, we define
\begin{equation*}
\Int(A) :=\{f \in B[X] \mid f(A) \subseteq A\}
\end{equation*}
and
\begin{equation*}
\Int_K(A) := \Int(A) \cap K[X] = \{f \in K[X] \mid f(A) \subseteq A\}.
\end{equation*}
We will also require that $A \cap K = D$; this assumption is equivalent to the containment $\Int_K(A)\subseteq\Int(D)$.
\begin{Def}\label{Standard assumptions}
When $A$ is a torsion-free $D$-algebra such that $A \cap K = D$, we say that $A$ is a $D$-algebra with \textit{standard assumptions}. When $A$ is finitely generated as a $D$-module, we say that $A$ is of \textit{finite type}.
\end{Def}

With these definitions, it is clear that $\Int_K(A)$ is always a subring of the commutative ring $K[X]$. The algebraic structure of $\Int(A)$ is more difficult to analyze. It is straightforward to verify that $\Int(A)$ is closed under addition, and in fact has the structure of a left $\Int_K(A)$-module. However, because $B[X]$ may contain polynomials with non-commuting coefficients, there is no guarantee that $\Int(A)$ is closed under multiplication. Indeed, let $g, h \in \Int(A)$ and let $f = gh$ be the product of $g$ and $h$ in $B[X]$. Then, we have $g(a), h(a) \in A$ for all $a \in A$, but because $f(a)$ need not equal $g(a)h(a)$, it is not clear whether or not $f(a)$ is in $A$. Thus, we arrive at an important question: is $\Int(A)$ a ring when $B$ is noncommutative?

There are cases where $\Int(A)$ has been proved to be closed under multiplication, and thus has a ring structure under the usual operations inherited from $B[X]$. For instance, this will be true if $A$ itself is a commutative ring. More generally, if each element of $A$ is a sum of units and central elements, then $\Int(A)$ is a ring \cite[Thm. 1.2]{WerMat}. In particular, this theorem applies when $A = M_n(D)$, the algebra of $n \times n$ matrices with entries in $D$, because $M_n(D)$ has a $D$-module basis consisting of invertible matrices. The condition in \cite[Thm. 1.2]{WerMat} is sufficient for $\Int(A)$ to be a ring, but is not necessary; counterexamples may be found in \cite[Ex. 3.8]{Wer1} and in \cite{FriTriangular}. To date, no example has been given of a noncommutative $D$-algebra $A$ for which $\Int(A)$ is not a ring. 

As mentioned in Example \ref{M_n(K) example}, Frisch proved in \cite[Thm. 7.2]{Fri1} that the rings $\Int(M_n(D))$ and $M_n(\Int_K(M_n(D)))$ are isomorphic. This result led the second author to search for other algebras with a similar property \cite{Wer}. To do this, the problem was recast in the following way. Assume that $A$, as a $D$-module, is free of finite rank, with $D$-basis $\alpha_1, \ldots, \alpha_t$. Then, when does $\alpha_1, \ldots, \alpha_t$ form a basis for $\Int(A)$ as an $\Int_K(A)$-module? With this formulation, Frisch's theorem shows that
\begin{equation}\label{IntMnD1}
\Int(M_n(D)) = \bigoplus_{1 \leq i, j \leq n} \Int_K(M_n(D)) E_{ij}
\end{equation}
where the $E_{ij}$ are the standard matrix units that form a $D$-basis of $M_n(D)$. An algebra $A = \bigoplus_i D \alpha_i$ such that 
\begin{equation}\label{Decomp def eq}
\Int(A) = \bigoplus_i \Int_K(A) \alpha_i
\end{equation}
is called $\textit{Int}_K$-\textit{decomposable} with respect to $\{\alpha_i\}_i$. 
By \cite[Prop. 1.4]{Wer}, this property is independent of the $D$-basis chosen for $A$. Thus, a free $D$-algebra $A$ such that \eqref{Decomp def eq} holds is called simply $\Int_K$-decomposable. We will use the adjective \textit{Int-decomposable} (with no subscript $K$) when we wish to speak of these algebras collectively, without reference to a specific domain $D$ or base field $K$.

The main theorem of \cite{Wer} proved that there is a close connection between Int-decomposable algebras and direct sums of matrix algebras. In \cite[Thm. 6.1]{Wer} it is shown that for $D$ a Dedekind domain with finite residue rings and $A$ a free $D$-algebra of finite rank, $A$ is $\Int_K$-decomposable if and only if for each nonzero prime $P$ of $D$, there exist $n, t \in \N$ and a finite field $\F_q$ such that $A/PA \cong \bigoplus_{i=1}^t M_n(\F_q)$. 

The work in \cite{Wer} depended crucially on the presence of a $D$-module basis for $A$, and it was desirable to know if $\Int$-decomposable algebras could be defined and studied without assuming that $A$ was free. This is indeed possible, and doing so is the focus of the current paper. The key insight was to notice that an $\Int_K$-decomposable algebra is one for which $\Int(A)$ can be generated (as a subring of $B[X]$) by $A$ and $\Int_K(A)$, and this property can be precisely expressed in terms of tensor products of $D$-algebras. We say that $A$ is $\Int_K$-decomposable if and only if $\Int(A) \cong \Int_K(A) \otimes_D A$ (see Definition \ref{Int-decomp def}); the only limitations we impose on $A$ are our standard assumptions that $A$ is torsion-free and $A \cap K = D$. Note that for the case of the full matrix algebra $A=M_n(D)$ considered initially by Frisch, the matrix ring $M_n(\Int_K(M_n(D)))$ is canonically isomorphic to the ring $\Int_K(M_n(D))\otimes_D M_n(D)$.

If $D$ is Dedekind with finite residue rings and $A$ is finitely generated as a $D$-module, then we are able to extend the classification given by \cite[Thm. 6.1]{Wer} (Theorem \ref{Int-decomp classification}). Hence, $\Int_K$-decomposable algebras are those which are residually a direct sum of copies of a matrix ring over a finite field. Moreover, we are able to obtain two alternate characterizations of $\Int$-decomposability, one in terms of the completions of $A$ at primes of $D$ (Theorem \ref{Int_K-decomposability Theorem}) and the other in terms of the extended $K$-algebra $B = K \otimes_D A$ (Theorem \ref{Int-decomposable algebras over number fields}). We are also able to describe when $\Int_K(A)= \Int(D)$ (Theorem \ref{When Int_K(A) = Int(D)}), which answers a question raised in \cite{Fri2}. 

In Section \ref{Int-decomp Sec}, we state the more general definition of $\Int_K$-decomposable and prove several of the theorems mentioned above. Section \ref{Int-decomp Completion Section} discusses the classification of $\Int_K$-decomposable algebras in terms of completions. As our work will show, $\Int_K$-decomposable algebras are related to matrix algebras via their residue rings and completions, but the two types of algebras are not the same. Theorem \ref{Implications} clarifies this situation by presenting various counterexamples. 

We close the paper by studying the consequences of Theorems \ref{Int-decomp classification} and \ref{Int_K-decomposability Theorem}. An easy corollary of our classification theorems is that an $\Int_K$-decomposable algebra $A$ must be a maximal $D$-order in the $K$-algebra $B$, and $B$ must be a semisimple $K$-algebra. Along these lines, in Section \ref{Extended classification}, we use the theory of maximal orders (as presented in \cite{Reiner}) to establish the last part of our classification. In Theorem \ref{Int-decomposable algebras over number fields} we prove that if $D$ is the ring of integers of a number field $K$, then $A$ is $\Int_K$-decomposable if and only if the following four conditions are satisfied: $B$ is a separable $K$-algebra with simple components which have the same center $F$; $F$ is a finite unramified Galois extension of $K$; the simple components of $B$ are unramified at each finite place of $F$; and $A$ is a maximal $D$-order in $B$.

From this general theorem, we obtain two relevant corollaries. First, if $A$ is the ring of integers of a finite extension $L$ of $K$, then $A$ is $\Int_K$-decomposable if and only if $L$ is an unramified Galois extension of $K$ (Corollary \ref{Int-decomp alg int}). Second, if $D = \Z$, then $A$ is $\Int_\Q$-decomposable if and only if for some $n$, $A$ is isomorphic to a finite direct sum of copies of $M_n(\Z)$ (Corollary \ref{IntQ decomposable algebras}). This last result implies that if $\Int_\Q(A) = \Int(\Z)$, then $A$ is isomorphic to a finite direct sum of copies of $\Z$.

\section{Int-decomposable Algebras}\label{Int-decomp Sec}

We begin by recalling the definition of $\Int_K$-decomposability in the case of free $D$-algebras which was given in \cite{Wer}.

\begin{Def}\label{Int-decomp 1 def}(\cite[Def. 1.2]{Wer})
Let $A$ be a $D$-algebra that, as a $D$-module, is free of finite rank, so that $A = \bigoplus_{i=1}^t D \alpha_i$ for some $D$-module basis $\{ \alpha_1, \ldots, \alpha_t\}$. We say that $A$ is \textit{$Int_K$-decomposable} with respect to $\{\alpha_i\}_{i=1}^t$ if $\Int(A) = \bigoplus_{i=1}^t \Int_K(A) \alpha_i$ as an $\Int_K(A)$-module. 
\end{Def}

It is shown in \cite[Prop. 1.4]{Wer} that the $\Int_K$-indecomposability of $A$ does not depend on the $D$-module basis $\{\alpha_1, \ldots, \alpha_t\}$. That is, $A$ is $\Int_K$-decomposable with respect to one basis if and only if it is $\Int_K$-decomposable with respect to every basis. Thus, we can---and will---say algebras are $\Int_K$-decomposable without referring to a specific basis.

A useful way to interpret Definition \ref{Int-decomp 1 def} is the following. Assume $A = \bigoplus_i D \alpha_i$. Then, we have $B = \bigoplus_i K \alpha_i$ and it follows that any $f \in B[X]$ can be expressed (uniquely) in the form $f = \sum_i f_i \alpha_i$, where each $f_i \in K[X]$. If $f \in \Int(A)$ and $A$ is $\Int_K$-decomposable then we may conclude that each $f_i \in \Int_K(A)$. This property can sometimes be used to quickly show that an algebra is not Int-decomposable. For example, let $D = \Z$ and $A = \Z[i]$, the Gaussian integers. Then, $\frac{(1+i)(X^2-X)}{2} \in \Int(A)$, but $\frac{(X^2-X)}{2} \notin \Int_\Q(A)$; hence, $A = \Z[i]$ is not $\Int_\Q$-decomposable. As we shall see in Corollary \ref{Int-decomp alg int}, this is related to the fact that 2 is a ramified prime of $\Z[i]$.

The most prominent examples of $\Int_K$-decomposable algebras are the matrix rings $M_n(D)$. As mentioned in the introduction, Frisch proved in \cite[Thm. 7.2]{Fri1} that when $A = M_n(D)$, there is a $D$-algebra isomorphism between $M_n(\Int_K(A))$ and $\Int(A)$, as in (\ref{IntMnD1}). Some rings of algebraic integers and certain quaternion algebras can also be Int-decomposable \cite[Sec. 6]{Wer}. 

The main theorem of \cite{Wer} shows that $\Int_K$-decomposable algebras can be recognized by the structure of their residue rings $A/PA$, where $P$ runs through the primes of $D$. 

\begin{Thm}\label{Old characterization}(\cite[Thm. 6.1]{Wer})
Let $D$ be a Dedekind domain with finite residue rings. Let $A$ be a free $D$-algebra of finite rank with standard assumptions. Then, $A$ is $\Int_K$-decomposable if and only if for each nonzero prime $P$ of $D$, there exist $n, t \in \N$ and a finite field $\F_q$ such that $A/PA \cong \bigoplus_{i=1}^t M_n(\F_q)$.
\end{Thm}

Both Definition \ref{Int-decomp 1 def} and the proof of Theorem \ref{Old characterization} depended on the presence of a $D$-basis for $A$. Our first goal in this paper is to generalize the definition of $\Int_K$-decomposable so that it applies to algebras that are not necessarily free. We will then go on to show (Theorem \ref{Int-decomp classification}) that Theorem \ref{Old characterization} still holds under this more general definition.

\begin{Def}\label{Int-decomp def}
Let $D$ be an integral domain and $A$ a torsion-free $D$-algebra. Consider the following $D$-bilinear map:
\begin{align*}
\Int_K(A) \times A &\to \Int(A)\\
(f(X),a) &\mapsto f(X)\cdot a
\end{align*}
By the universal property of the tensor product, there exists a unique $D$-module homomorphism 
\begin{align}\label{Phi}
\Phi:\Int_K(A)\otimes_D A&\to \Int(A)
\end{align}
which maps every elementary tensor product $f(X)\otimes a$ to $f(X)\cdot a$.  We say that $A$ is \textit{$Int_K$-decomposable} if $\Phi$ is an isomorphism of $D$-modules, so that  $\Int(A) \cong \Int_K(A) \otimes_D A$.

Recall that the tensor product $\Int_K(A)\otimes_D A$ has a natural $D$-algebra structure with multiplication given by $(f_1(X)\otimes a_1)(f_2(X)\otimes a_2)=(f_1(X)f_2(X))\otimes( a_1a_2)$ (see \cite[Chapt. III, \S 4, n. 1]{BourbakiAlg}) , for all $f_i\in\Int_K(A), a_i\in A$, $i=1,2$. Moreover, since the elements of $K$ are central in $B$, the map $\Phi$ above induces a $D$-algebra structure on $\Int(A)$ (and so, $\Phi$ becomes a $D$-algebra homomorphism). Therefore, when $A$ is $\Int_K$-decomposable, $\Phi$ is an isomorphism of $D$-algebras.
\end{Def}

This definition has a number of immediate consequences, among them that Definition \ref{Int-decomp def} reduces to the original Definition \ref{Int-decomp 1 def} when $A$ is free.

\begin{Prop}\label{Int-decomp properties}\mbox{}
\begin{enumerate}[(1)]
\item If $A$ is $\Int_K$-decomposable, then $\Int(A)$ is a ring.
\item The following are equivalent:
\begin{enumerate}[(i)]
\item $A$ is $\Int_K$-decomposable.
\item the $D$-module $\Int(A)$ satisfies the universal property of the tensor product of $\Int_K(A)$ and $A$.
\item $\Int(A)$ is equal to the subring of $B[X]$ generated by $\Int_K(A)$ and $A$.
\end{enumerate}
\item Assume $A = \bigoplus_{i=1}^t D \alpha_i $ is free of finite rank as a $D$-module. Then, $A$ is $\Int_K$-decomposable in the sense of Definition \ref{Int-decomp 1 def} if and only if $A$ is $\Int_K$-decomposable in the sense of Definition \ref{Int-decomp def}. \end{enumerate}
\end{Prop}
\begin{proof}
(1) This is a generalization of \cite[Prop. 2.2]{Wer}. When $A$ is $\Int_K$-decomposable, it is actually isomorphic as a ring to the $D$-algebra $\Int_K(A) \otimes_D A$.

(2) The equivalence of (i) and (ii) is clear. For (iii), note that for any $A$ ($\Int_K$-decomposable or not), the module $\Int(A)$ contains all products of the form $f(X) \cdot a$, where $f \in \Int_K(A)$ and $a \in A$. Since $\Int(A)$ is closed under addition, it also contains the subring of $B[X]$ generated by $\Int_K(A)$ and $A$. Given the definition of $\Phi$, $A$ being $\Int_K$-decomposable is equivalent to $\Int(A)$ equaling this subring.

(3) Since $A = \bigoplus_i D \alpha_i$, we have the chain of equalities
\begin{equation*}
\Int_K(A)\otimes_D A=\Int_K(A)\otimes_D(\bigoplus_i D\alpha_i)=\bigoplus_i (\Int_K(A)\otimes_D D\alpha_i)=\bigoplus_i \Int_K(A)\alpha_i
\end{equation*}
from which the equivalence of the definitions is clear.
\end{proof}

\begin{Rem}\mbox{}
\begin{itemize}
\item While $A$ being $\Int_K$-decomposable implies that $\Int(A)$ is a ring, the converse is not true. There are numerous examples of $D$-algebras $A$ which are not $\Int_K$-decomposable but still $\Int(A)$ is a ring. For instance, when $G$ is a finite group and $A$ is the group algebra $DG$, $\Int(A)$ is a ring by \cite[Thm. 1.2]{WerMat}. However, whenever the characteristic of $D/P$ divides $|G|$, the group ring $A/PA \cong (D/P)G$ is not semisimple \cite[Cor. 3.4.8]{MiliesSehgal}, and hence cannot satisfy Theorem \ref{Old characterization}. Thus, the group algebra $DG$ will not be $\Int_K$-decomposable in such cases.

Also, if $p$ is an odd prime of $\Z$, $D = \Z_{(p)}$, and $A$ is a free $D$-algebra of finite rank, then for each $k > 0$ the residue ring $A/p^kA$ has odd order. It then follows from \cite[Thms. 2.4, 3.7]{Wer1} that $\Int(A)$ is a ring; but certainly $A$ can be chosen so that $A$ is not $\Int_\Q$-decomposable.

Finally, let $T_n(D)$ be the $D$-algebra of $n\times n$ upper triangular matrices. Frisch has recently shown \cite{FriTriangular} that $\Int(T_n(D))$ is a ring. But, $T_n(D)$ does not satisfy the condition of Theorem \ref{Old characterization}, so $T_n(D)$ is not $\Int_K$-decomposable.

\item According to Proposition \ref{Int-decomp properties}, Frisch's result (\ref{IntMnD1}) can be restated as follows: 
$$\Int(M_n(D))=\Int_K(M_n(D))\otimes_D M_n(D).$$
\end{itemize}
\end{Rem}

When $D$ is a Dedekind domain, we can prove that the map $\Phi$ in Definition \ref{Int-decomp def} is always injective.

\begin{Lem}\label{Phi injective}
Let $D$ be a Dedekind domain and $A$ a $D$-algebra of finite type with standard assumptions. Then, $\Phi:\Int_K(A)\otimes_D A\to \Int(A)$ is injective.
\end{Lem}
\begin{proof}
Define $\Psi: K[X] \otimes_D A \to B[X]$ by $\Psi(f(X) \otimes_D a) = f(X) a$. Then, $\Psi$ is an isomorphism of $K$-algebras. Since $D$ is Dedekind and $A$ is torsion-free, it follows that $A$ is a projective $D$-module, hence flat. Therefore, the containment $\Int_K(A)\subseteq K[X]$ implies that $\Int_K(A)\otimes_D A\subseteq K[X]\otimes_D A$. Thus, the map $\Phi$ is the restriction of $\Psi$ to $\Int_K(A)\otimes_D A$, and since $\Psi$ is injective so is $\Phi$. 
\end{proof}




In this way, for $D$ Dedekind, we may identify $\Int_K(A)\otimes_D A$ with the subring of $B[X]$ generated by $\Int_K(A)$ and $A$. We use this fact in proving the next proposition.

\begin{Prop}\label{IntKA=IntD implies A Int-decomposable}
Let $D$ be a Dedekind domain and $A$ a $D$-algebra of finite type with standard assumptions. If $\Int_K(A)=\Int(D)$ then $A$ is $\Int_K$-decomposable.
\end{Prop}
\begin{proof}
Since $D$ is Dedekind, $\Int(A)$ contains $\Int_K(A) \otimes_D A$ via the map $\Phi$ by Lemma \ref{Phi injective}. For the other containment, assume $\Int_K(A) = \Int(D)$ and let $\Int(D, A)$ be the set $\Int(D, A) := \{f \in B[X] \mid f(D) \subseteq A\}$. Clearly, $\Int(A) \subseteq \Int(D, A)$. By \cite[Prop. IV.3.3]{CaCh} we have $\Int(D,A) = \Int(D)\otimes_D A$, so that
\begin{equation*}
\Int(A) \subseteq \Int(D, A) = \Int(D)\otimes_D A = \Int_K(A) \otimes_D A,
\end{equation*}
as required. 
\end{proof}

We can also generalize \cite[Thm. 3.3]{Wer} and prove that $\Int_K$-decomposability is a local property when $A$ is finitely generated. Note that this will hold without the assumption that $D$ is Dedekind.

\begin{Prop}\label{IntKdeclocal}
Let $D$ be a domain and let $A$ be a $D$-algebra of finite type with standard assumptions. Then $A$ is $\Int_K$-decomposable if and only if $A_P$ is $\Int_K$-decomposable for each prime $P$ of $D$.
\end{Prop}
\begin{proof}
By definition, $A$ is $\Int_K$-decomposable if and only if the map $\Phi$ in (\ref{Phi}) is an isomorphism. By \cite[Prop. 3.9]{AtMc}, this holds if and only  the $D$-modules $\Int(A)$ and $\Int_K(A)\otimes_D A$ are isomorphic locally at each prime ideal $P$ of $D$, that is, the induced maps
\begin{equation*}
\Phi_P:(\Int_K(A) \otimes_D A) \otimes_D D_P \to \Int(A)\otimes_D D_P
\end{equation*}
are isomorphisms for each prime $P$ of $D$. 

Recall that for a $D$-module $M$ and a multiplicative set $S \subset D$, we have $S^{-1} M \cong M \otimes_D S^{-1} D$. Thus, we always have $A_P \cong A \otimes_D D_P$ and $\Int_K(A)_P \cong \Int_K(A) \otimes_D D_P$. Since $A$ is finitely generated, by \cite[Prop. 3.2]{Wer} we have $\Int_K(A)_P = \Int_K(A_P)$ and $\Int(A_P) = \Int(A)_P$. Hence, $\Int_K(A) \otimes_D D_P \cong \Int_K(A_P)$ and $\Int(A) \otimes_D D_P \cong \Int(A_P)$.

Using this and other standard properties of tensor products (as in \cite[Chap. II, \S 5]{BourbakiAlg}), we have
\begin{align*}
(\Int_K(A) \otimes_D A) \otimes_D D_P &\cong \Int_K(A) \otimes_D (A\otimes_D D_P)\\
&\cong \Int_K(A)\otimes_D A_P\\
&\cong (\Int_K(A)\otimes_D D_P) \otimes_{D_P} A_P\\
&\cong \Int_K(A_P)\otimes_{D_P} A_P.
\end{align*}
Hence, the induced map $\Phi_P$ is an isomorphism if and only if
\begin{equation*}
\Int(A_P) \cong (\Int_K(A) \otimes_D A) \otimes_D D_P \cong \Int_K(A_P)\otimes_{D_P} A_P,
\end{equation*}
which means that $A_P$ is $\Int_K$-decomposable.
\end{proof}

We can now extend the classification of $\Int$-decomposable algebras given in Theorem \ref{Old characterization}.

\begin{Thm}\label{Int-decomp classification}
Let $D$ be a Dedekind domain with finite residue rings. Let $A$ be a $D$-algebra of finite type with standard assumptions. Then, $A$ is $\Int_K$-decomposable if and only if for each nonzero prime $P$ of $D$, there exist $n, t \in \N$ and a finite field $\F_q$ such that $A/PA \cong \bigoplus_{i=1}^t M_n(\F_q)$.

In particular, if $A$ is commutative, then $A$ is $\Int_K$-decomposable if and only if for each $P$ there exists a finite field $\F_q$ such that $A/PA \cong \bigoplus_{i=1}^t \F_q$ for some $t \in \N$.
\end{Thm}
\begin{proof}
Since $D$ is Dedekind and $A$ is finitely generated and torsion-free, $A$ is a projective $D$-module \cite[Cor. p. 30]{Nark}. Hence, for each prime $P$, $A_P$ is free as a $D_P$-module, and $A_P$ has finite rank because $A$ is finitely generated. Applying Theorem \ref{Old characterization} to $A_P$, we see that $A_P$ is $\Int_K$-decomposable if and only if $A_P/PA_P \cong \bigoplus_i M_n(\F_q)$ for some $n$ and some $\F_q$. Using Proposition \ref{IntKdeclocal} and the fact that $A/PA \cong A_P/PA_P$, we obtain the stated theorem.
\end{proof}

Under the same hypotheses on $D$, we can describe those $A$ for which $\Int_K(A) = \Int(D)$. This generalizes \cite[Thm. 4.6]{Wer}, which dealt with the case where $A$ was free.

\begin{Thm}\label{When Int_K(A) = Int(D)}
Let $D$ be a Dedekind domain with finite residue rings. Let $A$ be a $D$-algebra of finite type with standard assumptions. Then, $\Int_K(A) = \Int(D)$ if and only if for each nonzero prime $P$ of $D$, $A/PA \cong \bigoplus_{i=1}^t D/P$ for some $t \in \N$.
\end{Thm}
\begin{proof}
We know that $\Int(D) = \bigcap_P \Int(D)_P$ (where the intersection is over nonzero primes $P$ of $D$) and that $\Int(D)_P = \Int(D_P)$ for each $P$. The analogous equalities for $\Int_K(A)$ are shown in \cite[Props. 3.1, 3.2]{Wer}. Thus, $\Int_K(A) = \Int(D)$ if and only if $\Int_K(A_P) = \Int(D_P)$ for each $P$. But, as in Theorem \ref{Int-decomp classification}, each $A_P$ is a free $D_P$-module of finite rank. By \cite[Thm. 4.6]{Wer}, $\Int_K(A_P) = \Int(D_P)$ if and only if $A_P / PA_P \cong \bigoplus_i D_P/PD_P$. The result now follows because $A_P/PA_P \cong A/PA$ and $D_P / PD_P \cong D/P$.
\end{proof}

\section{Int-decomposable Algebras via Completions}\label{Int-decomp Completion Section}
In this section, we provide an alternate characterization of Int-decomposable algebras. As in Section \ref{Int-decomp Sec}, we assume that $D$ is a Dedekind domain with finite residue fields and that $A$ is a $D$-algebra of finite type with standard assumptions. Theorem \ref{Int-decomp classification} asserts that $A$ is $\Int_K$-decomposable precisely when for each nonzero prime $P$ of $D$, $A/PA$ is isomorphic to a direct sum of copies of a matrix ring with entries in a finite field. Instead of focusing on $A/PA$, we can work with the $P$-adic completion $\whA_P=\varprojlim A/P^k A$ of $A$, which in this case is isomorphic to $\whD_P \otimes_D A$ (where $\whD_P$ is the $P$-adic completion of $D$) because $A$ is of finite type. In Theorems \ref{Int_K-decomposability Theorem} and \ref{Int_K(A) = Int(D) completions}, we prove that both Int-decomposability and the equality $\Int_K(A) = \Int(D)$ can be characterized in terms of the completions $\whA_P$. These results complement Theorems \ref{Int-decomp classification} and \ref{When Int_K(A) = Int(D)}, respectively.

Recall first the following definition.

\begin{Def}\label{Null ideal}
For a ring $R$ and an $R$-algebra $A$, the \textit{null ideal} of $A$ with respect to $R$, denoted $N_R(A)$, is the set of polynomials in $R[X]$ that kill $A$. That is, $N_R(A) = \{ f \in R[X] \mid f(A) = 0\}$. When $R$ is commutative, $N_R(A)$ is easily seen to be an ideal of $R[X]$. If $R=A$, we set $N_R(A)=N(A)$.
\end{Def}

Null ideals are a useful tool for dealing with integer-valued polynomials because there is a correspondence between the elements of $\Int_K(A)$ and the null ideals $N_{D/dD}(A/dA)$, where $d \in D$. Specifically, let $f(X) = g(X)/d \in K[X]$, where $g(X) \in D[X]$ and $d \in D$. Then, $f \in \Int_K(A)$ if and only if the residue of $g$ in $(D/dD)[X]$ is in $N_{D/dD}(A/dA)$. 

Int-decomposability can be expressed in terms of null ideals (this was the main strategy employed in \cite{Wer}; see \cite[Def. 4.3, Thm. 4.4]{Wer}).  To do this, we need a notion of ``decomposability'' for $N(A/P^k A)$. This is accomplished in the next definition, which is the analog of Definition \ref{Int-decomp def}.

\begin{Def}\label{Null decomposability}
Let $P$ be a nonzero prime of $D$ and let $k > 0$. We say that $A/P^{k}A$ is \textit{$N_{D/P^k}$-decomposable} if the canonical ring isomorphism $(D/P^k)[X]\otimes_{D/P^k}A/P^kA \cong (A/P^k A)[X]$ that maps each elementary tensor product $f(X) \otimes a$ to $f(X) \cdot a$ induces the following isomorphism of $D/P^k$ modules:
\begin{equation*}
N_{D/P^k}(A/P^k A)\otimes_{D/P^k}A/P^k A \cong N(A/P^k A).
\end{equation*}
\end{Def}

\begin{Lem}\label{Decomposability equivalence}
Let $D$ be a Dedekind domain with finite residue rings. Let $A$ be a $D$-algebra of finite type with standard assumptions. Then, $A$ is $\Int_K$-decomposable if and only if $A/P^kA$ is $N_{D/P^k}$-decomposable for each nonzero prime $P$ of $D$ and each $k > 0$.
\end{Lem}
\begin{proof}
By Proposition \ref{IntKdeclocal} we may localize at a prime $P$ and assume that $D$ is a discrete valuation ring, and hence that $A$ is free. Furthermore, each $A/P^k A$ is free as a $D/P^k$-module, so we can always find $\alpha_1, \ldots, \alpha_t \in A/P^k A$ such that $A/P^kA = \bigoplus_{i=1}^t D/P^k \alpha_i$. By Proposition \ref{Int-decomp properties} (3), $A$ is $\Int_K$-decomposable if and only if $A$ is $\Int_K$-decomposable in the sense of Definition \ref{Int-decomp 1 def}; and by \cite[Thm. 4.4, Def. 4.3]{Wer}, this is equivalent to having $N(A/P^kA) = \bigoplus_{i=1}^t N_{D/P^k}(A/P^kA)\alpha_i$ for each $k > 0$. Proceeding as in Proposition \ref{Int-decomp properties} (3), one may show this last condition is equivalent to having $N_{D/P^k}(A/P^k A)\otimes_{D/P^k}A/P^k A \cong N(A/P^k A)$.
\end{proof}

By Proposition \ref{IntKdeclocal}, $\Int_K$-decomposability is a local property. So, in this section we will often reduce to the local case, namely that $D$ is a discrete valuation ring (DVR) with finite residue field $\mathbb{F}_{q_0}$ and maximal ideal $P=\pi D$. Moreover, notice that when $P$ is a maximal ideal of $D$, we have $A/PA\cong A_P/PA_P$, so that 
\begin{equation*}
N_{D/P}(A/PA)=N_{D_P/PD_P}(A_P/PA_P).
\end{equation*}
We will use this fact freely in our subsequent work. In order to ease the notation, we set $A_k=A/P^kA$ and $D_k=D/P^k$, for each $k\in\N$. Note that $D_k\subseteq A_k$ and that $A_k$ is a torsion-free $D_k$-algebra, which is finitely generated as a $D_k$-module.

\begin{Lem}\label{decomposition Ak}(\cite[Thm. 5.10]{Wer})
Let $k \in \N$. Then, $A_k$ is $N_{D_k}$-decomposable if and only if $A_k\cong\bigoplus_{i=1}^t M_n(T_k)$ for some $n, t \in \N$ and a finite commutative local ring $T_k$ with principal maximal ideal $\mfm_k$ which is generated by the same uniformizer $\pi$ of $D$, so that $\mfm_k=\pi T_k$.
\end{Lem}
Recall that a commutative ring is \emph{chain ring} if its set of ideals is totally ordered by inclusion. In particular, a ring $T_k$ as described in Lemma \ref{decomposition Ak} is a chain ring, as in \cite{McD}. Note that $T_1$ is equal to a finite field $\F_q$, which contains the residue field of $D$ at $P$.

In Lemma \ref{decomposition Ak}, the $A_k$'s form an inverse system with respect to the natural projection maps $A_k\to A_k/\pi^{k-1}A_k\cong A_{k-1}$,  and these maps are compatible with finite direct product and matrix rings. In particular, we have
\begin{equation*}
T_k\to T_k/\pi^{k-1}T_k\cong T_{k-1}
\end{equation*}
so the $T_k$'s also form an inverse system of chain rings. Moreover, since the nilpotency of $\pi$ in $A_k$ is $k$, it follows that the nilpotency of $\pi$ in $T_k$ is also $k$ and the residue field of $T_k$ is $T_1=\F_q$.

Given that the $P$-adic completion $\whA_P$ is equal to the inverse limit $\varprojlim A_k$, it is natural to consider the inverse limit of the chain rings $T_k$. It is well known that the completion of any DVR $V$ with maximal ideal $\mathfrak m$ is realized as the inverse limit of the chain rings $V/\mathfrak m^k$, $k\in\N$. The next lemma shows that, under certain mild assumptions, the converse is also true (it is probable that this lemma is a known result, but a proof was not found in the available literature, so one is provided for the sake of the reader).

\begin{Lem}\label{inverse limit chain rings}
Let $\{T_k\}_{k\in\N}$ be an inverse system of chain rings with maximal ideals $\mfm_k=\pi_k T_k$ such that $k$ is the nilpotency of $\pi_k$, and the transition maps $\theta_k:T_k\to T_{k-1}$ are all surjective (so, without loss of generality, we may assume that $\pi_k\mapsto \pi_{k-1}$). We assume that $T_1$ is the common residue field of the rings $T_k$.

Then the inverse limit $\whT = \varprojlim T_k$ is a complete DVR with residue field isomorphic to $T_1$.
\end{Lem}
\begin{proof}
We identify the inverse limit $\whT$ with the subset of coherent sequences of the direct product of the $T_k$'s:
\begin{equation*}
\whT = \varprojlim T_k = \{(a_k)\in\prod_{k\geq 1}T_k \mid a_{k+1}\mapsto a_{k},\forall k\geq1\}.
\end{equation*}
Let
\begin{equation*}
\whm = \{(a_k)\in \whT \mid a_{k}\in \mfm_k,\forall k\geq1\}.
\end{equation*}
Note that, for $(a_k)\in \whT$, $(a_k)$ is in $\whm$ if and only if for some $k$ we have $a_k \in \mfm_k$. Clearly, $\whm$ is an ideal of $\whT$. We claim that every element of $\whT \setminus \whm$ is invertible, which shows that $\whT$ is a local ring with maximal ideal $\whm$. Indeed, let $(a_k)\in \whT \setminus \whm$. Then for each $k\in\N$, $a_k$ is invertible in $T_k$, and it is easy to see that $(a_k^{-1})$ is a coherent sequence and is the inverse of $(a_k)$ in $\whT$.

Moreover, $\whm = \pi \whT$, where $\pi = (\pi_k) \in \whT$. In fact, by definition $\pi \in \whm$. Conversely, if $(a_k)\in \whm$, then, for all $k\in\N$, we have $a_k=\pi_k b_k$, for some $b_k\in T_k$. One may verify that $(b_k)$ is a coherent sequence, so $(a_k) = \pi(b_k) \in \pi \whT$. Note also that $\pi$ is not a nilpotent element of $\whT$, because if $e\in\N$ is such that $\pi^e=(\pi_k^e)=0$, then we have $e\geq k$, for all $k\in\N$, a contradiction.

Now, we clearly have $\bigcap_{k\geq 1}\whm^k=(0)$, so by \cite[Chap. VI, \S 1, n. 4, Prop. 2]{BourbakiAlgComm} $\whT$ is a DVR. It remains to show that $\whT$ is complete with respect to the $\whm$-adic topology. 

Since each transition map $\theta_k:T_k\to T_{k-1}$ is surjective, each projection $\psi_k:\widehat{T}\to T_k$ is also surjective. Furthermore, the kernel of $\psi_k$ is $\pi^k \whT$, since the maximal ideal of $T_k$ has nilpotency $k$ by assumption. So, we may identify each chain ring $T_k$ with the residue ring $\whT/\pi^k \whT = \whT/\whm^k$. Hence, $\whT = \varprojlim \whT/\whm^k$, which shows that the topology on $\whT$ as the inverse limit of the chain rings $\{T_k\}_{k\in\N}$ coincides with the $\whm$-adic topology and that $\whT$ is complete with respect to the $\whm$-adic topology \cite[Chap. III, \S 2, n. 6]{BourbakiAlgComm}.
\end{proof}

The next theorem gives the promised characterization of $\Int_K$-decomposable algebras in terms of the completions $\whA_P$. 
Given a prime ideal $P$ of $D$, we denote by $\whD_P$ the $P$-adic completion of $D$, and denote by $\whK_P$ the $P$-adic completion of $K$ (which is also the fraction field of $\whD_P$).

\begin{Thm}\label{Int_K-decomposability Theorem}
Let $D$ be a Dedekind domain with finite residue rings. Let $A$ be a $D$-algebra of finite type with standard assumptions. Then, $A$ is $\Int_K$-decomposable if and only if, for 
each nonzero prime ideal $P\subset D$, there exist $n, t \in\N$ such that
\begin{equation}\label{decomposition AP}
\whA_P\cong\bigoplus_{i=1}^t M_n(\whT_P)
\end{equation} 
where $\whT_P$ is a complete DVR with finite residue field and quotient field which is a finite unramified extension of $\whK_P$. 

In particular, if $A$ is commutative then $A$ is $\Int_K$-decomposable if and only if, for each nonzero prime ideal $P\subset D$, $\whA_P\cong\bigoplus_{i=1}^t \whT_P$, where  $\whT_P$ is as above.
\end{Thm}
\begin{proof}
Without loss of generality, we may suppose that $D$ is a DVR with maximal ideal $P=\pi D$. We retain the notation introduced at the beginning of this section. 

Suppose first that $A$ is $\Int_K$-decomposable. Then, the $P$-adic completion $\whA_P$ of $A$ is equal to the inverse limit of the rings $A_k$ \cite[Chap. III, \S 2, n. 6]{BourbakiAlgComm}. For each $k\in\N$, let $T_k$ be as in Lemma \ref{decomposition Ak},  and let $\whT_P$ be their inverse limit, which is a complete DVR with maximal ideal $\whm_P$ by Lemma \ref{inverse limit chain rings}. Since the formation of inverse limit commutes with finite direct sums, by Lemma \ref{decomposition Ak} we have
\begin{equation}\label{AP direct sum matrix rings}
\widehat{A}_P=\varprojlim_{k\geq 1}A_k=\bigoplus_{i=1}^t\varprojlim_{k\geq 1}M_n(T_k)=\bigoplus_{i=1}^t M_n(\widehat{T}_P).
\end{equation}

Now, $D_k \subseteq A_k \cong \bigoplus_i M_n(T_k)$, so by \cite[Lem. 5.1]{Wer} each matrix ring $M_n(T_k)$ contains a (central) copy of $D_k$, which is contained in the set of scalar matrices $T_k$. Since the maximal ideals of $D_k$ and $T_k$ have the same generator $\pi$, $D_k\subseteq T_k$ is an unramified extension of chain rings \cite[p. 281]{McD}. So, by Lemma \ref{inverse limit chain rings}, $\whD_P = \varprojlim D_k\subseteq \whT_P= \varprojlim T_k$ is an unramified extension of complete DVRs since $\pi$ generates the maximal ideals of both $\whD_P$ and $\whT_P$. Let $\whF_P$ be the quotient field of $\whT_P$, let $\F_q$ be the residue field of $\whT_P$, and let $\F_{q_0}$ be the residue field of $\whD_P$. Then, $\whF_P/\whK_P$ is an unramified field extension of finite degree $[\whF_P : \whK_P] = [\F_q : \F_{q_0}]$. This establishes that if $A$ is $\Int_K$-decomposable, then we have the desired decomposition (\ref{decomposition AP}) for $\whA_P$.

Conversely, suppose (\ref{decomposition AP}) holds, where $\whT_P$ is a complete DVR which is a finite unramified extension of $\whD_P$. In particular, $\whP \cdot \whT_P = \whm$, the maximal ideal of $\whT_P$. As above, let $\F_q$ be the (finite) residue field of $\whT_P$. Then,  
\begin{equation*}
\whA/\whP\whA \cong A/PA \cong \bigoplus_{i=1}^t M_n(\F_q)
\end{equation*}
so by Theorem \ref{Int-decomp classification} $A$ is $\Int_K$-decomposable.
\end{proof}

Theorem \ref{Int_K-decomposability Theorem} is the analog of Theorem \ref{Int-decomp classification}. There is also an analogous form of Theorem \ref{When Int_K(A) = Int(D)}, the proof of which requires the next lemma.

\begin{Lem}\label{DVR lemma}
Let $D$ be a DVR with maximal ideal $P = \pi D$. Let $A$ be a $D$-algebra with standard assumptions, and let $\whA$ be the $P$-adic completion of $A$. Then, $\Int_K(\whA) = \Int_K(A)$.
\end{Lem}
\begin{proof}
The containment $\Int_K(\whA) \subseteq \Int_K(A)$ is clear, since $A$ embeds in $\whA$. Conversely, let $f \in \Int_K(A)$ and $\alpha \in \whA$. Suppose $f(X) = g(X)/\pi^k$, where $g \in D[X]$ and $k \in \N$. If $k = 0$, then $f \in D[X] \subseteq \Int_K(\whA)$, so assume that $k > 1$.

Via the canonical projection $\whA \to A/\pi^k A$, we see that there exists $a\in A$ such that $\alpha \equiv a \pmod{\pi^k \whA}$. Since the coefficients of $g$ are central in $A$, we get $g(\alpha)\equiv g(a) \pmod{\pi^k \whA}$. Thus, $f(\alpha)=f(a)+\lambda/\pi^k$, where $\lambda\in\pi^k \whA$, so that $f(\alpha)\in \whA$. Hence, $f \in \Int_K(\whA)$ and $\Int_K(\whA) = \Int_K(A)$.
\end{proof} 

\begin{Thm}\label{Int_K(A) = Int(D) completions}
Let $D$ be a Dedekind domain with finite residue rings. Let $A$ be a $D$-algebra of finite type with standard assumptions. Then, $\Int_K(A)=\Int(D)$ if and only if, for each nonzero prime ideal $P$ of $D$, $\whA_P\cong\bigoplus_{i=1}^t \widehat{D}_P$, for some $t\in \N$.
\end{Thm}
\begin{proof}
Note that if $t$ is the rank of $A$, defined as the dimension of $B=K\otimes_D A$ over $K$, then, for each prime ideal $P$ of $D$, the rank of $A_P=D_P \otimes_D A\subset B$ is equal to $t$, so that $t$ does not depend on the particular prime ideal $P$. As we have already remarked, $A_P$ is free as a $D_P$-module, so that $A_P\cong \bigoplus_{i=1}^t D_P$ (as $D_P$-modules).

We may work locally since $\whA_P=(\widehat{A_P})_{PD_P}$. So, we will assume that $D$ is a DVR and we will omit the suffix $P$.

If $\Int_K(A)=\Int(D)$ then $A$ is $\Int_K$-decomposable, by Proposition \ref{IntKA=IntD implies A Int-decomposable}. Hence, by Theorem \ref{When Int_K(A) = Int(D)}, we have $A/PA\cong \bigoplus_{i=1}^t D/P$. By Theorem \ref{Int_K-decomposability Theorem}, the completion $\whA$ decomposes as a finite direct sum of matrix rings over  a complete DVR $\whT$, which is a finite unramified extension of $\whD$ (so that $\whP \cdot \whT$ is equal to the maximal ideal $\whm$ of $\whT$). Since $\whA/\whP\whA \cong A/PA$, formula (\ref{decomposition AP}) becomes $\whA \cong \bigoplus_{i=1}^t \whD$, that is, $n=1$ and $\whD = \whT$, because the residue field $\whT/\whm$ of $\whT$ must be isomorphic to $D/P$, the residue field of $\whD$.

Conversely, if $\whA$ is isomorphic to a finite direct sum of copies of $\whD$, then by Lemma \ref{DVR lemma} we have 
\begin{equation*}
\Int_K(A)=\Int_K(\whA)=\Int_K(\bigoplus_i \whD)=\Int_K(\whD)=\Int(D)
\end{equation*}
as desired.
\end{proof}

At this point, it is apparent that Int-decomposable algebras are related to matrix algebras via their residue rings and completions. At the close of this paper, we will prove (Corollary \ref{IntQ decomposable algebras}) that when $D = \Z$, $A$ is $\Int_\Q$-decomposable if and only if there exist $n, t \in \N$ such that $A \cong \bigoplus_{i=1}^t M_n(\Z)$ (which implies that if $A$ is $\Int_\Q$-decomposable, then $\Int_\Q(A) = \Int_\Q(M_n(\Z))$ for some $n$). However, in general Int-decomposable algebras need not be direct sums of matrix algebras. We end this section with a theorem that considers some of the conditions that link matrix algebras with Int-decomposable algebras, and examines the implications among these conditions.

\begin{Thm}\label{Implications}
Let $D$ be a Dedekind domain with finite residue rings and $A$ a $D$-algebra of finite type with standard assumptions. Consider the following four conditions:
\begin{enumerate}[(i)]
\item \label{cond1} there exists $n \in \N$ such that $A \cong M_n(D)$.
\item \label{cond2} there exists $n\in\N$ such that $\widehat{A}_P\cong M_n(\whD_P)$, for all primes $P$ of $D$.
\item \label{cond3} $A$ is $\Int_K$-decomposable.
\item \label{cond4} there exists $n \in \N$ such that $\Int_K(A) = \Int_K(M_n(D))$.
\end{enumerate}
Then, the following implications hold: $(\ref{cond1}) \Rightarrow (\ref{cond2})$, $(\ref{cond2}) \Rightarrow (\ref{cond3})$, and $(\ref{cond2}) \Rightarrow (\ref{cond4})$; but for none of these three implications does the converse hold. Finally, $(\ref{cond3}) \not\Rightarrow (\ref{cond4})$ and $(\ref{cond4}) \not\Rightarrow (\ref{cond3})$.
\end{Thm}

\begin{proof}

The implication $(\ref{cond1}) \Rightarrow (\ref{cond2})$ is clear, and $(\ref{cond2}) \Rightarrow (\ref{cond3})$ is Theorem \ref{Int-decomp classification}.\\

($(\ref{cond2}) \Rightarrow (\ref{cond4})$) 
Clearly, it is sufficient to prove the statement locally at each maximal ideal $P$ of $D$. Thus, we suppose that $D$ is a DVR. By Lemma \ref{DVR lemma} we have $\Int_K(A)=\Int_K(\widehat{A})$ and $\Int_K(M_n(D))=\Int_K(M_n(\widehat{D}))$, and since $\whA\cong M_n(\whD)$ the statement holds.\\

We now show by counterexamples that the other stated implications do not hold.\\

($(\ref{cond2}) \not\Rightarrow (\ref{cond1})$) Let $p$ be an odd prime of $\Z$, let $D = \Z_{(p)}$, and let $A$ be the standard quaternion algebra $A = D \oplus D\mathbf{i} \oplus D\mathbf{j} \oplus D\mathbf{k}$ (so that $\mathbf{i}^2 = \mathbf{j}^2 = -1$ and $\mathbf{i}\mathbf{j} = \mathbf{k} = -\mathbf{j}\mathbf{i}$). Then, it is well known (cf.\ \cite[Exer. 3A]{Goodearl}) that $A/p^k A \cong M_2(D/p^k D) \cong M_2(\Z/p^k \Z)$ for all $k > 0$, so that $\widehat{A}\cong M_2(\Z_p)$. However, $A \not\cong M_2(D)$, because (among other reasons) $A$ contains no nonzero nilpotent elements.\\

($(\ref{cond3}) \not\Rightarrow (\ref{cond2})$) Take $A = M_n(D) \oplus M_n(D)$. Then, $A$ is $\Int_K$-decomposable by Theorem \ref{Int-decomp classification}, but $A/PA \cong M_n(D/P) \oplus M_n(D/P)$, so (\ref{cond2}) does not hold.\\

($(\ref{cond4}) \not\Rightarrow (\ref{cond2})$) Again take $A = M_n(D) \oplus M_n(D)$. Then, (\ref{cond2}) does not hold, but $\Int_K(A) = \Int_K(M_n(D))$ by \cite[Thm. 2.3]{Wer}.\\

($ (\ref{cond3}) \not\Rightarrow (\ref{cond4})$) Let $K \subsetneqq L$ be an unramified Galois extension of number fields. Let $D = O_K$ and take $A = O_L$. Then---as we will show in Corollary \ref{Int-decomp alg int}---$A$ is $\Int_K$-decomposable. However, we argue that for all $n \in \N$ we have $\Int_K(O_L) \ne \Int_K(M_n(D))$. First, if $n = 1$, then $\Int_K(O_L) = \Int_K(D) = \Int(O_K)$; but, by Corollary \ref{Int-decomp alg int}, this is impossible because $L \ne K$. So, assume that $n >1$. Then, by \cite[Prop. 7]{PerWer} $\Int_K(O_L)$ is integrally closed (this also follows from \cite[Thm. 3.7]{LopWer}). But, $\Int_K(M_n(D))$ is never integrally closed when $n > 1$ \cite[Cor. 3.4]{PerWer1}, so $\Int_K(O_L) \ne \Int_K(M_n(D))$ for $n > 1$.\\

($(\ref{cond4}) \not\Rightarrow (\ref{cond3})$) Let $n > m \geq 1$ and take $A = M_n(D) \oplus M_m(D)$. Then, $\Int_K(M_n(D)) \subseteq \Int_K(M_m(D))$, so $\Int_K(A) = \Int_K(M_n(D))$ by \cite[Thm. 2.3]{Wer}. But, for any prime $P$ of $D$, we have $A/PA \cong M_n(D/P) \oplus M_m(D/P)$, which does not satisfy  Theorem \ref{Int-decomp classification}.
\end{proof}

\section{Extended Algebras and Maximal Orders}\label{Extended classification}
In this final section, we examine the consequences of Theorems \ref{Int-decomp classification} and \ref{Int_K-decomposability Theorem}, which allows us to give a global characterization of Int-decomposable algebras. The descriptions given in Theorems \ref{Int-decomp classification} and \ref{Int_K-decomposability Theorem} show that an $\Int$-decomposable algebra can be described---either residually or in terms of its completions---in terms of matrix rings. In the case where $D$ is the ring of integers of a number field $K$, we are able to completely classify $\Int_K$-decomposable algebras $A$ as those which are the maximal orders of the extended $K$-algebra $B$; $B$ is a separable $K$-algebra whose simple components share a common center $F$; $F$ is a finite unramified Galois extension of $K$; and each simple component of $B$ is unramified at each finite place of $F$.

The work in this section relies heavily on the theory of maximal orders, as presented in \cite{Reiner}. We begin by recalling several definitions from \cite{Reiner} and \cite{Rowen}. As in earlier sections, $D$ denotes a Dedekind domain with fraction field $K$.

\begin{Defs}\label{Noncomm defs}\mbox{}
Let $B$ be a finite dimensional $K$-algebra. 
\begin{itemize}
\item By the Wedderburn Structure Theorem \cite[Thm. 3.5]{Pierce}, if $B$ is semisimple, then we have $B = \bigoplus_{i=1}^r M_{n_i}(\mcD_i)$, for some uniquely determined $r,n_i\in\N$ and division rings $\mcD_i$;  the $B_i = M_{n_i}(\mcD_i)$ are the \emph{simple components} of $B$. We denote by $Z(\mcD_i)$ the center of $\mcD_i$, which is a finite field extension of $K$. We say that $B$ is \emph{separable} if $B$ is a finite dimensional semisimple $K$-algebra, such that the center of each simple component of $B$ is a separable field extension of $K$ \cite[p. 99]{Reiner}. 

\item A \textit{$D$-order} in $B$ is a subring $A$ of $B$ such that $A$ is a finitely generated $D$-submodule of $B$ and $K \cdot A = B$. A \textit{maximal $D$-order} in $B$ is a $D$-order that is not properly contained in any other $D$-order of $B$ (see \cite[p. 108, 110]{Reiner}). Note that in the setting of this paper, a $D$-algebra $A$ of finite type is a $D$-order in the extended $K$-algebra $B=K\otimes_D A$ (and vice versa, a $D$-order $A$ in a $K$-algebra $B$ is a $D$-algebra of finite type).

\item We say that a field extension $F/K$ is a \emph{splitting field} of the $K$-algebra $B$ if the extended $F$-algebra $B\otimes_K F$ is a direct sum of full matrix algebras over $F$, that is, $B\otimes_K F\cong\bigoplus_{i=1}^s M_{n_i}(F)$ \cite[Def. 18.30]{Rowen}. It is easy to see that if a finite dimensional $K$-algebra $B$ admits a splitting field $F$, then $B$ is semisimple, since the extended $F$-algebra $B\otimes_K F$ is semisimple (cf.\ \cite[p. 151]{Rowen}).

\item If $B=M_n(\mcD)$ is a $K$-central simple algebra, where $\mcD$ is a division algebra, we denote by $\deg(B)=\sqrt{[B:K]}$ the \emph{degree} of $B$. If $D$ is a Dedekind domain and $P\subset D$ a maximal ideal, then $\whK_P$ is a splitting field of $B$ if and only if $B$ is \emph{unramified} at $P$ in the sense of \cite[Chap. 8, \S 32]{Reiner}, that is, $\whB_P=B\otimes_K \whK_P\cong M_{nd}(\whK_P)$, where $d$ is the degree of $\mcD$.
\end{itemize}
\end{Defs}

When $D$ is the ring of integers of a number field, we will demonstrate that $\Int_K$-decomposable algebras can be completely classified in terms of these definitions. We first consider the local case where $D$ is a DVR, and then globalize this to the general case.

\subsection{Local case}
In this subsection, $D$ is a DVR with maximal ideal $P$ and finite residue field. As in Section \ref{Int-decomp Completion Section}, we denote by $\whD$ and $\whA$ the $P$-adic completions of $D$ and $A$ (and in general all completions are with respect to the $P$-adic topology).

\begin{Thm}\label{local case}
Let $A$ be a $D$-algebra of finite type with standard assumptions and let $B=A\otimes_D K$ be the extended $K$-algebra. Then $A$ is $\Int_K$-decomposable if and only if $A$ is a maximal order in $B$, $B$ is a separable $K$-algebra with simple components $B_1,\ldots,B_r$ and there exists  a finite unramified extension $\whF$ of $\whK$ and $n\in\N$ which satisfy these conditions  for each $i=1,\ldots,r$:
\begin{enumerate}[(i)]
\item $F_i\otimes_K \whK\cong\prod_{j=1}^{k_i} \whF$ for some $k_i\in\N$.
\item  $B_i\otimes_{F_i}\whF=M_n(\whF)$.
\end{enumerate} 
\end{Thm}

Note that, by \cite[Chap. 6, Prop. 6.1]{Nark} the above condition (i) is equivalent to the following: for each $i=1,\ldots,r$, all the prime ideals in the integral closure $D_{F_i}$ of $D$ in $F_i$ (which necessarily lie above $P$) are unramified and have the same residue field degree, equal to $[\whF:\whK]$ (which is independent of $i$). In particular, $\whF$ is the completion of $F_i$ at each prime ideal of $D_{F_i}$. The second condition says that $\whF$ is a splitting field of each simple component $B_i$, that is, $B_i$ is unramified at each finite place of its center $F_i$ and the degree of each simple component $B_i$ as a $F_i$-central simple algebra is constant, independent of $i$. So we can say that $P$ is unramified in each $B_i$.

\begin{proof}
$(\Leftarrow)$ Assume the above conditions on $A,B$ are satisfied. Since $B$ is semisimple,  we have
\begin{equation}\label{B}
B=\bigoplus_{i=1}^r M_{n_i}(\mcD_i)
\end{equation}
for some $r,n_i\in\N$ and division rings $\mcD_i$.  We denote by $B_i$ the simple component $M_{n_i}(\mcD_i)$ of $B$ and by  $F_i$ the center of $\mcD_i$, for $i=1,\ldots,r$.

Note that for each $i=1,\ldots,r$ by condition (i) we have
\begin{align}\label{completion simple components}
B_i \otimes_K \whK = (B_i \otimes_{F_i} F_i) \otimes_K \whK &= B_i\otimes_{F_i} (F_i \otimes_K \whK)\nonumber\\
&= B_i \otimes_{F_i} (\prod_{j=1}^{k_i} \whF) = \prod_{j=1}^{k_i} (B_i \otimes_{F_i} \whF)
\end{align}
so by condition (ii) we conclude that $B_i\otimes_K \whK=\prod_{j=1}^{k_i}M_n(\whF)$. Hence, the $P$-adic completion of $B$ is:
\begin{equation}\label{completion of B 1}
\whB \cong \bigoplus_{i=1}^r (B_i\otimes_K \whK) = \bigoplus_{i=1}^r \bigoplus_{j=1}^{k_i}M_n(\whF) =\bigoplus_{h=1}^t M_n(\whF).
\end{equation}

Finally, since $B$ is a separable $K$-algebra and $A$ is a maximal $D$-order in $B$, it follows by \cite[Thm. 11.5]{Reiner} that $\whA$ is a maximal $\whD$-order in $\whB$. Moreover, by \cite[Thm. 10.5]{Reiner}, $A$ decomposes as  $A=\bigoplus_{i=1}^r A_i$, where $A_i$ is a maximal $D$-order in $B_i$, for all $i=1,\ldots,r$. Similarly, $\whA$ decomposes as $\whA=\bigoplus_{i=1}^t \whA_i$, where each $\whA_i$ is a maximal $\whD$-order in the simple component $M_n(\whF)$ of $\whB$. By \cite[Thm. 17.3]{Reiner}, each $\whA_i$ is conjugated by a unit of $M_n(\whF)$ to the maximal $\whD$-order $M_n(\whT)$, where $\whT$ is the DVR of the local field $\whF$. In particular, each maximal $\whD$-order  $\whA_i$ is isomorphic to $M_n(\whT)$, so by Theorem \ref{Int_K-decomposability Theorem} $A$ is $\Int_K$-decomposable.

$(\Rightarrow)$ Assume that $A$ is $\Int_K$-decomposable. By Theorem \ref{Int_K-decomposability Theorem} we have that $\whA = A \otimes_D \whD \cong \bigoplus_{i=1}^t M_n(\whT)$ for some $n,t\in\N$ and a finite unramified extension $\whT$ of $\whD$. Then,
\begin{equation}\label{completion of B}
\whB=B\otimes_K \whK=(A\otimes_D K)\otimes_D \whD=\whA\otimes_D K\cong\bigoplus_{h=1}^t M_n(\whF)
\end{equation} 
(note that $\whT\otimes_D K=\whF $). Therefore, $\whB$ is a $\whK$-semisimple algebra with center equal to $Z(\whB)\cong\bigoplus_{i=1}^t \whF$. Since an unramified extension of a local field is separable \cite[Thm 5.26]{Nark}, $\whB$ is a separable $\whK$-algebra. Moreover, each component $M_n(\whT)$ of $\whA$ is a maximal $\whD$-order in the respective simple component $\whB_i=M_n(\whF)$ of $\whB$  (see \cite[Thm. 8.7]{Reiner}), so, by \cite[Thm. 10.5]{Reiner}, $\whA$ is a maximal $\whD$-order in $\whB$. By \cite[Thm. 11.5]{Reiner}, $A$ is a maximal $D$-order in $B$.

Now, by \cite[Prop. 10.6b]{Pierce} $B$ is $K$-separable, hence semisimple. Therefore, $B$ decomposes as a finite direct sum of matrix algebras $B_i=M_{n_i}(\mcD_i)$ as in (\ref{B}) for some $r,n_i\in\N$ and division rings $\mcD_i$ whose centers $F_i=Z(\mcD_i)$ are finite separable field extensions of $K$.  

By \cite[Prop. 6.1]{Nark}, for each $i=1,\ldots,r$, $F_i\otimes_K \whK=\prod_{j=1}^{k_i}\whF_{ij}$, where $\whF_{ij}$ is a finite separable extension of $\whK$ for all $i,j$ (the $\whF_{ij}$ are the completions of $F_i$ at the different prime ideals of $D_{F_i}$ which lie above $P$). Moreover, $B_i\otimes_{F_i}\whF_{ij}$ is a central simple algebra over $\whF_{ij}$, say equal to $M_{m_{ij}}(\widehat{\mcD}_{ij})$, where $\widehat{\mcD}_{ij}$ is a central division algebra over $\whF_{ij}$ \cite[Cor. 7.8]{Reiner}. As in (\ref{completion simple components}) we get
\begin{equation}\label{completion of B 2}
\whB \cong \bigoplus_{i=1}^r (B_i\otimes_K \whK) = \bigoplus_{i=1}^r \bigoplus_{j=1}^{k_i}(B_i\otimes_{F_i}\whF_{ij})=\bigoplus_{i=1}^r \bigoplus_{j=1}^{k_i}M_{m_{ij}}(\widehat{\mcD}_{ij}).
\end{equation}
By comparing (\ref{completion of B}) and (\ref{completion of B 2}) and applying the Wedderburn Structure Theorem, we deduce that $m_{ij}=n$, $\widehat{\mcD}_{ij}=\whF$ for all $i$ and $j$, and $\sum_{i=1}^r k_{i}=t$. This forces $\whF_{ij}=\whF$ for all $i$ and $j$, so that condition (i) is satisfied. Also, $B_i\otimes_{F_i}\whF_{ij}=B_i\otimes_{F_i}\whF=M_n(\whF)$, so that condition (ii) is satisfied, too.
\end{proof}

\begin{Rem} Assume $A$ is $\Int_K$-decomposable, so $B$ decomposes as in (\ref{B}). For each $i=1,\ldots,r$, let $D_{F_i}$ be the integral closure of $D$ in the center $F_i$ of the simple component $B_i$ of $B$ and consider it as a (commutative) $D$-algebra. Then the theorem shows (via condition (ii)) that the $D_{F_i}$ are $\Int_K$-decomposable (according to Theorem \ref{Int-decomp classification}). It also shows that each component $A_i$ of $A$ is $\Int_K$-decomposable and also $\Int_{F_i}$-decomposable as well.

Note that while the degree of the $B_i$'s as $F_i$-central simple algebras is the same for all $i$, it is not necessarily true that the dimension of the $B_i$'s over $K$ is the same for all $i$. The point is that the centers $F_i$'s may be different from each other (and, in particular, have different degree over $K$). For example, let $D=\Z_{(p)}$ where $p$ is an odd prime, let $A_1$ be the standard quaternion algebra $A_1 = D \oplus D\mathbf{i} \oplus D\mathbf{j} \oplus D\mathbf{k}$ (so that $\mathbf{i}^2 = \mathbf{j}^2 = -1$ and $\mathbf{i}\mathbf{j} = \mathbf{k} = -\mathbf{j}\mathbf{i}$). Then, $B_1=\Q \oplus \Q\mathbf{i} \oplus \Q\mathbf{j} \oplus \Q\mathbf{k}$, so $n_1=1$ and $m_1=2$. Let $F/\Q$ be a quadratic field extension in which $p$ splits completely and let $D_{F,p}$ be the integral closure of $\Z_{(p)}$ in $F$ (so $D_{F,p}/pD_{F,p}\cong\Z/p\Z\times\Z/p\Z$). Let $A_2=M_2(O_{F,p})$, so $B_2=M_2(F)$, $n_2=2$, and $m_2=1$. Then $A=A_1\oplus A_2$ is $\Int_K$-decomposable. Note that $n_1 m_1=n_2 m_2=2$ and $B_1$ and $B_2$ have different dimension over $\Q$: $[B_1:\Q]=4$ but $[B_2:\Q]=8$.

In the global case where $D$ is the ring of integers of a number field, which will be treated in the next subsection, we will see that if condition (ii) of Theorem \ref{local case} holds at each maximal ideal of $D$, then the simple components of the separable $K$-algebra $B$ have all the same center $F$, which is an unramified Galois extension of $K$.
\end{Rem}

\begin{Rem} Let $A$ be an $\Int_K$-decomposable algebra, as in the statement of Theorem \ref{local case}. Since the finite unramified extension $\whF$ of $\whK$ of condition (ii) of the statement is a Galois extension, from (\ref{completion of B 1}) it is easy to see that $\whF$ is a splitting field of the $K$-algebra $B$.  However, in general the converse does not hold, that is, a finite unramified extension $\whF$ of $\whK$ can be a splitting field of $B$ without $A$ being $\Int_K$-decomposable. For example, let $F$ be a finite field extension of $K$ such that the maximal ideal $P$ of $D$ is unramified in $D_F$ and the  prime ideals above $P$ have different residue field degree. Then, by Theorem \ref{Int-decomp classification} the $D$-algebra $A=D_F$ is not $\Int_K$-decomposable. Let $\whF$ be a finite unramified extension of $\whK$ containing all the completions of $A$ at the different prime ideals above $P$. Then $F\otimes_K \whF=(F\otimes_K \whK)\otimes_{\whK} \whF=\prod_{j}\widehat{F}_j\otimes_{\whK}\whF=\prod_j \whF$ (notice that $\widehat{F}_j\otimes_{\whK}\whF$ is equal to a direct product of copies of $\whF$), so $\whF$ is a splitting field of $B=F$.
\end{Rem}

\subsection{Global case}

We now establish a global variant of Theorem \ref{local case}. In this final subsection, we assume that $D$ is the ring of integers of a number field $K$. This enables us to use some of the powerful tools of algebraic number theory, such as the Tchebotarev Density Theorem and the Hasse-Brauer-Noether-Albert Theorem. As usual, $A$ is a $D$-algebra of finite type with standard assumptions.

\begin{Thm}\label{Int-decomposable algebras over number fields}
Let $K$ be a number field with ring of integers $D$. Let $A$ be a $D$-algebra of finite type with standard assumptions and let $B=A\otimes_D K$ be the extended $K$-algebra.  Then, $A$ is $\Int_K$-decomposable if and only if $A$ is a maximal order in $B$ and $B$ is a separable $K$-algebra with simple components $B_1, \ldots, B_r$ such that the following hold:
\begin{enumerate}[(i)]
\item the $B_i$ share a common center $F$.
\item $F$ is a finite unramified Galois field extension of $K$.
\item for each $i$, $B_i$ is unramified at each finite place of $F$.
\item the degree of $B_i$ as an $F$-central simple algebra is the same for each $i$.
\end{enumerate}
\end{Thm}
\begin{proof}
$(\Leftarrow)$ Assume the above conditions on $A$ and $B$ are satisfied. Let $P$ be a fixed maximal ideal of $D$. Then, by conditions (ii) and (iii) and \cite[Prop. 6.1]{Nark} we have $F\otimes_K \whK=\prod_{j=1}^{k}\whF_P$, for some $k\in\N$, where $\whF_P$ is a finite unramified extension of $\whK_P$; note that $\whF_P$ is the completion of $F$ at any prime ideal which lies above $P$. Therefore, by (\ref{completion simple components}) and conditions (i) and (iv) we have
\begin{equation*}
\whB_P=B\otimes_K \whK_P=\bigoplus_{i=1}^r (B_i\otimes_K \whK_P)=\bigoplus_{i=1}^r \bigoplus_{j=1}^{k}(B_i\otimes_{F}\whF_P)=\bigoplus_{h=1}^t M_n(\whF_P)
\end{equation*}
Since $A$ is a maximal $D$-order in $B$, by \cite[Cor. 11.2 \& Thm. 11.5]{Reiner} $\whA_P$ is a maximal $\whD_P$-order in $\whB_P$, so by \cite[Thm. 10.5]{Reiner} $\whA_P$ decomposes as $\whA_P=\bigoplus_{h=1}^t \whA_h$, where each $\whA_h$ is a maximal $\whD_P$-order in $M_n(\whF_P)$. If $\whT_P$ is the valuation ring of $\whF_P$, then by \cite[Thm. 17.3]{Reiner} $\whA_h$ is a conjugate of (hence isomorphic to) $M_n(\whT_P)$, for each $h=1,\ldots,t$. Since $P$ was an arbitrary maximal ideal of $D$, by Theorem \ref{Int_K-decomposability Theorem} it follows that $A$ is $\Int_K$-decomposable.

$(\Rightarrow)$ Assume now that $A$ is $\Int_K$-decomposable. Since $\Int_K$-decomposability is a local property (Proposition \ref{IntKdeclocal}), for each maximal ideal $P$ of $D$, $A_P=A\otimes_D D_P$ is $\Int_K$-decomposable, so we may apply Theorem \ref{local case} to $A_P$. Thus, $A_P$ is a maximal $D_P$-order in $B$, and $B$ is a separable $K$-algebra and decomposes as $B=\bigoplus_{i=1}^r B_i$ with simple components $B_i$ with centers $F_i$ which are finite separable field extensions of $K$, for $i=1,\ldots,r$. Moreover, there exists $n\in\N$ such that $[B_i:F_i]=n^2$, for each $i=1,\ldots,r$, and hence (iv) holds. Note that $r$ is independent of the particular maximal ideal $P$ of $D$. Moreover, since $A$ is locally a maximal $D_P$-order in $B$, $A$ is a maximal $D$-order in $B$ (\cite[Cor. 11.2]{Reiner}).

For each prime ideal $P$ of $D$, by condition (i) of Theorem \ref{local case} there exists a finite unramified extension $\whF_P$ of $\whK_P$ such that, for each $i=1,\ldots,r$, $\whF_P$ is the completion of $F_i$ at any prime ideal $Q$ of the ring of integers $D_{F_i}$ which lies over $P$. Furthermore, by condition (ii) of Theorem \ref{local case}  $\whF_P$ is a splitting field of the $F_i$-central simple algebra $B_i$. These facts imply that all the field extensions $F_1,\ldots,F_r$ are unramified over $K$ and by the Tchebotarev Density Theorem they also are Galois extensions (see \cite[Cor. VII.13.8]{Neu}). Moreover, since a finite Galois extension $F$ of $K$ is completely determined by  the set of prime ideals of $K$ which split completely in $F$ (again by the same theorem of Tchebotarev, see \cite[Cor. VII.13.10]{Neu}), it follows that $F_1,\ldots,F_r$ are all equal to the same finite unramified Galois extension $F$. All of this proves conditions (i), (ii), and (iii).
\end{proof}

We close this paper with two corollaries. In the first one we specialize Theorems \ref{Int-decomp classification} and \ref{When Int_K(A) = Int(D)} to the case where $D$ and $A$ are rings of integers in number fields, which results in a very clean description of $\Int_K$-decomposable algebras. In the second corollary, we show that over $\Q$, an $\Int_{\Q}$-decomposable algebra $A$ must be isomorphic to a finite direct sum of copies of $M_n(\Z)$, for some $n\in\N$. This corollary also demonstrates that---with our usual assumptions in place---a matrix algebra over $\Z$ can be recognized by its residues and completions.

\begin{Cor}\label{Int-decomp alg int}
Let $K\subseteq L$ be number fields with rings of integers $O_K$ and $O_L$, respectively. Consider $O_L$ as an $O_K$-algebra. Then
\begin{enumerate}[(1)]
\item $O_L$ is $\Int_K$-decomposable if and only if $L/K$ is an unramified Galois extension.
\item $\Int_K(O_L) = \Int(O_K)$ if and only if $L = K$.
\end{enumerate}
\end{Cor}
\begin{proof}
(1) This follows from Theorem \ref{Int-decomposable algebras over number fields}.

(2) Clearly, $L = K$ implies that $\Int_K(O_L) = \Int(O_K)$. So, assume that $\Int_K(O_L) = \Int(O_K)$. By Theorem \ref{When Int_K(A) = Int(D)}, $O_L / PO_L \cong \bigoplus_i O_K/P$ for each nonzero prime $P$ of $O_K$. From this, we see that $O_L$ is $\Int_K$-decomposable as an $O_K$-algebra, and moreover that $f(Q|P) = 1$ for each $P$ and each $Q$ above $P$ (note that the same conclusion can be obtained from Theorem \ref{Int_K(A) = Int(D) completions}). By (1), $L$ is an unramified Galois extension of $K$ with all inertial degrees equal to 1, and so $L = K$ by \cite[Thm. VI.3.8]{Neu}.
\end{proof}

\begin{Cor}\label{IntQ decomposable algebras}
Let $A$ be a $\Z$-algebra of finite type with standard assumptions. 
The following are equivalent.
\begin{enumerate}[(1)]
\item $A$ is $\Int_\Q$-decomposable
\item There exist $n, r \in \N$ such that $A \cong \bigoplus_{i=1}^r M_n(\Z)$
\item For all primes $p$, there exist $n, r \in \N$ such that $A/pA \cong \bigoplus_{i=1}^r M_n(\Z/p\Z)$
\item For all primes $p$, there exist $n, r \in \N$ such that $\whA_p \cong \bigoplus_{i=1}^r M_n(\Z_p)$
\end{enumerate}
In particular, if $A$ is $\Int_\Q$-decomposable, then $\Int_\Q(A) = \Int_\Q(M_n(\Z))$ for some $n$; and if $\Int_\Q(A) = \Int(\Z)$, then $A \cong \bigoplus_{i=1}^r \Z$.
\end{Cor}
\begin{proof}
It suffices to prove (1) $\Leftrightarrow$ (2). The other equivalences follow by Theorems \ref{Int-decomp classification} and \ref{Int_K-decomposability Theorem}.

By Theorem \ref{Int-decomposable algebras over number fields}, if $A$ is isomorphic to a finite direct sum of copies of $M_n(\Z)$, then $A$ is $\Int_{\Q}$-decomposable (the same conclusion follows by the application of either Theorem \ref{Int-decomp classification} or Theorem \ref{Int_K-decomposability Theorem}). Conversely, if $A$ is $\Int_{\Q}$-decomposable then by the same theorem $A$ is a maximal $\Z$-order in $B$ and $B$ is a separable $\Q$-algebra, say $B=\bigoplus_{i=1}^r B_i$, where the $B_i$'s are the simple components of $B$. Since there is no proper unramified extension of $\Q$, the center of each $B_i$ is equal to $\Q$, by condition (ii) of Theorem \ref{Int-decomposable algebras over number fields}.

Now, for each $i=1,\ldots,r$, by the Hasse-Brauer-Noether-Albert Theorem \cite[Thm. 32.11]{Reiner} either $B_i\cong M_n(\Q)$ or there are at least two primes of $\Q$ (finite or infinite) which ramify in $B_i$ (see \cite[Chap 8, \S 32.1, Exer. 1]{Reiner}). The latter situation cannot occur, because $B_i$ is unramified at each finite place of its center $\Q$ by condition (iii) of Theorem \ref{Int-decomposable algebras over number fields}. Since this holds for each $i=1,\ldots,r$, we must have $B=\bigoplus_{i=1}^r M_n(\Q)$ (the fact that $n$ is the same for each simple component of $B$ is a consequence of condition (iv) of Theorem \ref{Int-decomposable algebras over number fields}). Finally, since $\Z$ is a PID, each maximal $\Z$-order of $M_n(\Q)$ is conjugate (hence isomorphic) to $M_n(\Z)$, so $A\cong\bigoplus_{i=1}^r M_n(\Z)$, as desired. The implication that $\Int_\Q(A) = \Int_\Q(M_n(\Z))$ is clear, and the last claim follows from Proposition \ref{IntKA=IntD implies A Int-decomposable} and Theorem \ref{When Int_K(A) = Int(D)}.
\end{proof}

The following example shows that the conclusion of Corollary \ref{IntQ decomposable algebras} may fail if we work over a number field which is a proper extension of $\Q$.

\begin{Ex}
Let $K=\Q(\sqrt{5})$. Consider the standard quaternion algebra $B$ over $K$, which is $B= K \oplus K\mathbf{i} \oplus K\mathbf{j} \oplus K\mathbf{k}$, where $\mathbf{i}^2 = \mathbf{j}^2 = -1$ and $\mathbf{i}\mathbf{j} = \mathbf{k} = -\mathbf{j}\mathbf{i}$. This algebra is unramified at each finite prime of the ring of integers $D$ of $K$ but is ramified at the two infinite real primes. Hence, any maximal $D$-order $A$ in $B$ is $\Int_K$-decomposable, by Theorem \ref{Int-decomposable algebras over number fields}, but $A$ cannot be isomorphic to a direct sum of matrix rings because $B$ is a division ring.
\end{Ex}

\noindent{\bf Acknowledgments} This work has been supported by  the grant ``Assegni Senior'' of the University of Padova. The authors wish to thank Daniel Smertnig for pointing out the results which led to Corollary \ref{IntQ decomposable algebras}.


\end{document}